\documentclass[11pt]{amsart}
\usepackage{amsmath, amssymb, latexsym, amsthm, amsfonts, mathtools, xparse}
\usepackage{mathrsfs}
\usepackage{url}
\usepackage{hyperref}
\usepackage{xcolor}  
\usepackage[protrusion=false]{microtype}
\usepackage{geometry}
\mathtoolsset{showonlyrefs}

\usepackage[protrusion=false]{microtype}
\usepackage{parskip}
\usepackage{paralist}
\usepackage[T1]{fontenc}
\usepackage[backend=biber, style=alphabetic, giveninits]{biblatex}

\numberwithin{equation}{section}
\hypersetup{
    colorlinks=true, 
    linktoc=all,     
    linkcolor=blue,  
}
\theoremstyle{plain}
\newtheorem{thm}{Theorem}[section]
\newtheorem{lem}[thm]{Lemma}

\newcommand{\thmref}[1]{Theorem~\ref{#1}}
\newcommand{\lemref}[1]{Lemma~\ref{#1}}

\theoremstyle{definition}
\newtheorem{rmk}[thm]{Remark}

\newcommand{\mc}{\mathcal}
\newcommand{\mbf}{\mathbf}
\newcommand{\mbb}{\mathbb}
\newcommand{\mfk}{\mathfrak}
\newcommand{\md}{\text{ mod }}
\title[Determination of a pair of newforms]{Determination of a pair of newforms from the product of their twisted central values}
\author{Pramath Anamby }
\address{Department of Mathematics\\ 
Indian Institute of Science Education and Research\\ 
Pune -- 411008, India.}
\email{pramath.anamby@gmail.com}

\author{Ritwik Pal }
\address{Stat-Math Unit\\ 
Indian Statistical Institute\\ 
Kolkata -- 700108, India.}
\email{ritwik.1729@gmail.com}
\begin{document}
\maketitle
\begin{abstract}
    We show that a pair of newforms $(f,g)$ can be uniquely determined by the product of the central $L$-values of their twists. To achieve our goal,  we prove an asymptotic formula for the average of the product of the central values of two twisted $L$-functions- $L(1/2, f \times \chi)L(1/2, g \times \chi \psi)$, where $(f,g)$ is a pair of newforms. The average is taken over the primitive Dirichlet characters $\chi$ and $\psi$ of distinct prime moduli.
\end{abstract}
\section{Introduction}

Distinguishing two modular forms using the \textit{arithmetic} and \textit{analytic} data associated with them is an important problem in the theory of modular forms. The classical \textit{multiplicity-one} results use the arithmetic aspects, i.e., the Hecke eigenvalues and the Fourier coefficients to distinguish two Hecke newforms. On the other hand, it was shown in  \cite{luo1997determination} that the analytic data like the central values of twists of the corresponding $L$-functions can also be used for such purposes. This result has its roots in the Waldspurger's formula that relates the squares of the Fourier coefficients of a half integral weight modular form to the central values of the twists of the $L$-function associated with an elliptic newform.  In this direction, several interesting results dealing with quadratic twists (see \cite{munshi2010effective}), Rankin-Selberg convolutions (see \cite{ganguly2009determining}, \cite{munshi2015effective}) are available. This question has also been studied in the context of $GL(3)$-automorphic forms (see \cite{chinta2005determination}, \cite{munshi2015determination}). 

In this article, we address the question of distinguishing a pair of modular forms of arbitrary weight and level from the product of their central values. We do this by investigating the behaviour of a twisted average of central $L$-values over a suitable family of Dirichlet characters. We obtain an asymptotic formula for the twisted average of product of the central values of two $GL(2)$ $L$-functions. A crucial step in this direction is the estimation of a sum involving Hecke eigenvalues and a product of classical Kloosterman sums. In this regard, it is worth mentioning that in \cite{munshi2015determination}, an asymptotic for the twisted average of central values of $GL(3)$ $L$-functions was obtained by estimating a sum of $2$-dimensional Kloosterman sums. Using this asymptotic expression, we can also study the simultaneous non-vanishing of the central values in this family. The study of simultaneous non-vanishing of $L$-values has a rich history (see \cite{michel2002simultaneous}, \cite{ramakrishnan1average}, \cite{akbary2006simultaneous}, \cite{xu2011simultaneous}, \cite{khan2012simultaneous}, \cite{munshi2012note}, \cite{das2015simultaneous}). These non-vanishing results have several important implications in analytic number theory and the theory of automorphic forms (see for example \cite{luo1999generalized}, \cite{iwaniec2000non}, \cite{michel2002simultaneous}).

Let $S_k(N)$ denote the space of cusp forms of weight $k$ and level $N$. We prove the following unique determination result for a pair of Hecke newforms $(f,g)$ of different weights and levels.
\begin{thm}  \label{unipair}
For $i=1,2$, let $f_i\in S_{k_i}(N_i)$ and $g_i\in S_{l_i}(M_i)$ be Hecke newforms such that \begin{equation}
    L(1/2, f_1\times \chi)L(1/2, g_1\times \chi\psi)=L(1/2, f_2\times \chi)L(1/2, g_2\times \chi\psi)
\end{equation} 
for all primitive characters $\chi$ and $\psi$ of prime moduli. Then $k_1=k_2, \ l_1=l_2, \ N_1=N_2, \ M_1 =M_2$ and $f_1=f_2, \ g_1=g_2$.
\end{thm} 
To prove \thmref{unipair}, we first prove an asymptotic for a twisted average of product of central $L$-values. Before proceeding to state the asymptotic, let us set up the following notations. Let $f,g$ be a pair of Hecke newforms of weights $k_1, k_2$ and levels $N_1, N_2$, respectively. Let $Q_1$ be a positive real parameter and $q_2$ be any prime number such that $(q_2, N_2)=1$ and $q_2\asymp Q_1^c$ (for $c\ge 203/100$). We consider the family $\mc F=\{L(1/2, f\times \chi)L(1/2,g\times\chi\psi), (\chi \md q_1, \psi \md q_2): \chi, \psi \textit{ primitive }, q_1\in \mc D\}$, where $\mc D$ is the set defined as below:  
\begin{equation}
    \mc D:=\{q_1\text{ prime }: Q_1<q_1 <2 Q_1, (q_1, N_1N_2)=1\}.
\end{equation}
For a twisted average of the product of $L$-values $L(1/2, f\times \chi)L(1/2,g\times\chi\psi)$ in the family $\mc F$, we prove the following asymptotic.
\begin{thm}\label{thasymp}
    Let $f$, $g$ be two newforms of weight $k_1,k_2$ and level $N_1,N_2$ respectively. Then for any prime $p< \log Q_1$ with $(p, N_1N_2)=1$, we have
    \begin{equation}\label{sumS}
    \begin{split}
&\sum_{q_1\in \mc D}\sideset{}{^*}\sum_{\chi(q_1)}\sideset{}{^*}\sum_{\psi(q_2)}\overline{\chi}(p)\overline{\psi}(p)\overline{G(\chi)}^2L(1/2, f\times \chi)L(1/2, g\times \chi\psi) \\ &= \frac{\lambda_g(p)}{p^{1/2}} \pi_2(Q_1) q_2 +O(Q_1^{599/200+\epsilon} q_2)+ O(Q_1^{701/200 +\epsilon} q_2^{3/4 + \epsilon}),
\end{split}
\end{equation}
where the constants depend only on $\epsilon$, $f$ and $g$. Here $\sum^*$- denotes that the summation is over the primitive characters and $\pi_r(X)$ denotes the sum over prime powers, defined as below:
\begin{equation}
    \pi_r(X):=\sum\nolimits_{X<p<2X}p^r.
\end{equation}
\end{thm}
Using partial summation and the Prime number theorem it can be easily established that
\[\pi_r(X) \asymp \frac{X^r \pi(X)(2^{r+1}-1)}{r+1}.\]
Thus for a fixed $p$, such that $\lambda_{g}(p) \neq 0$, the main term in \eqref{sumS} is of the magnitude $Q_1^3 q_2 / \log Q_1$. With our choice of the parameters $Q_1$ and $q_2$ (i.e., $q_2 \gg Q_1^{203/100}$), for every small enough $\epsilon$, the main term is always bigger in magnitude than the error terms. Among the error terms, writing $q_2= Q_1^c$ for $c\geq 203/100$, we note that
\begin{equation*}
   Q_1^{599/200+\epsilon} q_2  \begin{cases}
   \ll Q_1^{701/200 +\epsilon} q_2 ^{3/4 + \epsilon}  & \text{ if } 203/100 \leq c \leq 51/25 ;\\
  \gg Q_1^{701/200 +\epsilon} q_2 ^{3/4 + \epsilon}  & \text{ if } c>51/25 .
   \end{cases}
\end{equation*}
Here we want to make a note of the fact that as per the methods of this article $c$ can be taken anything bigger than $101/50$ with a cost of different power-saving error terms than mentioned in \eqref{sumS}. However, for the sake of not complicating the exposition, we take $c \geq 203/100$. We also note that $p$ can be taken to be $\ll Q_1^\delta$ for some small $\delta>0$. But in order to avoid keeping track of an additional parameter during the course of calculation, we take $p\ll \log Q_1$.

To see that \thmref{unipair} follows from \thmref{thasymp}, we observe that by taking $q_2=Q_1^{5/2}+O(Q_1^2)$ and $Q_1 \rightarrow \infty$ in \thmref{thasymp} and invoking the strong multiplicity-one result for the elliptic newforms on $g_1$  and $g_2$, we can conclude $g_1=g_2$. After this, we prove a result about twisted average of central $L$-values over prime moduli (see \lemref{lemma}). Then we first  apply it on $g \times \chi$ to show that for any fixed primitive $\chi \bmod{q_1}$ there exists a primitive $\psi \bmod{q_2}$ such that $ L(1/2, g \times \chi \psi)$ is non-zero. This leads us to apply \lemref{lemma} again on $f_1$ and $ f_2$ separately. Eventually, by invoking the strong multiplicity-one theorem, we conclude $f_1=f_2$.

It is worth mentioning here that \cite[Theorem 5.1]{blomer2018second} has a stronger asymptotic result for a twisted average of central values, where the average is taken over the primitive characters of a single prime moduli. After taking $l=1, l'=p$ in the main term of \cite[Theorem 5.1]{blomer2018second}, we apply the Hecke relation for $\lambda_{f}(np)$ and $\lambda_{g}(np)$ in \cite[(5.2)]{blomer2018second} and write the main term in terms of polynomials in $\lambda_{f}(p)$ and $\lambda_{g}(p)$ (after a suitable truncation of $n$ in terms of $q$). We see that in the asymptotic expression of $\mathcal{Q}(f,g,1/2;1,p)$ in \cite[Theorem 5.1]{blomer2018second}, the error term in terms of $q$ is $q^{-1/144+\epsilon}$ and the main term has $p^{1/2}$ in the denominator. Thus to have the main term dominate the error term one needs to take $p \ll q^{1/72}$. Hence, the polynomials in $\lambda_f(p)$ and $\lambda_{g}(p)$ will be of degree at least $72$ due to the size of the error term in terms of $q$. Thus, to conclude the unique determination of a pair $(f,g)$ from their result boils it down to the following question.

\textbf{Question--} Let $f_1,g_1, f_2, g_2$ be holomorphic non-CM newforms. Let $P_1, P_2, P_3, P_4$  be four one variable polynomials of degree at least $72$, such that $P_1 (\lambda_{f_1}(p)) + P_2 (\lambda_{g_1}(p))= P_3 (\lambda_{f_2}(p))+P_4 (\lambda_{g_2}(p))$ for all but finitely many primes $p$. Then $(f_1,g_1)= (f_2,g_2)$. 
   
The solution to the above question does not seem to be trivial. Even a simpler version of the question after dropping $g_1, g_2$ and $P_2, P_4$ seems rather non-trivial. As unique determination of a pair $(f,g)$ is the main aim of this article, the Theorem 5.1 in \cite{blomer2018second} does not seem to be employable here.

The proof of Theorem \ref{thasymp} begins by expanding the left-hand side of \eqref{sumS} by using the approximate functional equations (AFEs) for $L(1/2, f\times \chi)$ and $L(1/2, g\times \chi\psi)$ (see section \ref{AFEsec}). Naturally, this transforms the problem into estimating an average over the Hecke-eigenvalues with  the character twists. The main term of the asymptotic comes from estimating certain character sums associated simultaneously with the $\lambda_{f}(1)$ and  $\lambda_{g}(p)$, that appear in the  \textit{dual} and the \textit{non-dual} parts of the approximate functional equations of $L(1/2, f\times \chi)$ and $L(1/2, g\times \chi\psi)$, respectively. 

The other parts  of the AFE, contribute to the error terms. By introducing the Gauss sum $\overline{G(\chi)}^2$, we can reduce the problem to dealing with the product of classical Kloosterman sums when estimating the error terms.
Otherwise, it would have resulted in estimating the averages of hyper-Kloosterman sums of degree $3$. A significant component of the error term estimation entails a meticulous analysis of the average of Kloosterman sums derived from the dual components of the AFEs of $L(1/2, f\times \chi)$ and $L(1/2, g\times \chi\psi)$. The additional sum over the character $\psi \md q_2$ plays a vital role in this estimate, without which the asymptotic would \textit{fail by $Q_1^{\epsilon}$}.

Another consequence of the asymptotic in Theorem \ref{thasymp} is  the following simultaneous non-vanishing result for the central values in the family $\mc F$. For any prime number $q_2$ such that $(q_2, N_2)=1$ and $q_2\asymp Q_1^c$ ($c\ge 203/100$), let $\mc M$ denote the set of characters given by
\begin{equation}\label{nonvanset}
\{(\chi \md q_1, \psi \md q_2): \chi, \psi \textit{ primitive }, q_1\in \mc D, L(1/2, f\times \chi)L(1/2, g\times \chi\psi)\neq 0\}.
\end{equation}
\begin{thm}\label{thsimul}
    Let $f$, $g$ be two newforms of weight $k_1,k_2$ and level $N_1,N_2$ respectively. Then we have
\begin{equation}\label{sizeM}
    \#\mc M\gg \frac{Q_1^{2} }{(\log Q_1)^{3}}q_2^{1-1/c-\epsilon},
\end{equation}
as $Q_1\rightarrow\infty$. The implied constants depend only on $f$ and $g$.
\end{thm} 
Let us call the ratio $\mc P=\frac{|\mc M|}{|\mc F|}$ as the proportion of non-vanishing. Then from Theorem \ref{thsimul} we can conclude that $\mc P\gg 1/q_2^{1/c+\epsilon}(\log Q_1)^2$. The estimate \eqref{sizeM} for the size of $\mc M$ is obtained by employing the large sieve inequality for Dirichlet characters and the asymptotic for the twisted average from Theorem \ref{thasymp}. An important point to note here is the choice of $p$ in Theorem \ref{thasymp}. Using the Prime number theorem, we show that it is possible to choose $p$ such that $\lambda_g(p)$ is greater than an absolute constant and with $p\ll_{k_2,N_2} 1$. This allows us to obtain the estimate \eqref{sizeM} for $\mc M$, which is independent of the prime $p$.

\subsection*{Acknowledgements.}
The authors thank NBHM, DAE for the financial support through NBHM postdoctoral fellowship. The first named author thanks IISER, Pune, where he is a postdoctoral fellow. The second named author thanks ISI, Kolkata, where he is a postdoctoral fellow. The authors thank Ritabrata Munshi for suggesting the problem and for various discussions. The authors also thank Soumya Das for comments and discussions.
\section{Preliminaries}
\subsection{Approximate functional equation:}\label{AFEsec} Let $h$ be a newform of weight $k$ and level $N$. Let $\chi$ be a Dirichlet character with conductor $q$ with $(q,N)=1$. Denote by $G(\chi)$, the Gauss sum associated with $\chi$. Then we have (see \cite[Theorem 5.3]{iwaniec2021analytic}
\begin{equation}\label{AFE}
   L(1/2 , h\times \chi)= \sum_{n\ge 1}\frac{\lambda_h(n)\chi(n)}{\sqrt{n}}V_h\left(\frac{nX}{q_{h,\chi}^{1/2}}\right)+\frac{G(\chi)^2}{q}\sum_{n\ge 1}\frac{\overline{\lambda_h(n)\chi(n)}}{\sqrt{n}}V_h\left(\frac{n}{Xq_{h,\chi }^{1/2}}\right),
\end{equation}
where $X>0$, $q_{h,\chi}=Nq^2$ and the smooth function $V_h(y)$ is given by
\begin{equation}
V_h(y)=\frac{1}{2\pi i}\int_{(3)}y^{-s}\frac{\gamma_h(\tfrac{1}{2}+s)}{\gamma_h(\tfrac{1}{2})}\frac{ds}{s}
\end{equation}
and it satisfies the following properties:
\begin{equation} \label{Vasymp}
    V_h(y)= 1+ O(y^{1/4}),
\end{equation}
\begin{equation} \label{Vbound}
    V_h(y) \ll (1+\frac{y}{k^2})^{-A}\quad \text{ for any } A>0.
\end{equation}
For the sake of simplicity, whenever the form $h$ is fixed and $\chi$ is varying, we adopt the notation $q_\chi:=q_{h,\chi}$.

In the rest of the article, the two terms in the right-hand side of \eqref{AFE} are referred to as \textit{non-dual} and $\textit{dual}$ parts of the approximate functional equation, respectively.

\subsection{Some character relations:} Let $r, n, m$ be any positive integer such that $(r,m)=1$ and $(n,m)=1$. Then we have
\begin{equation}\label{charsum}
\sideset{}{^*}{\sum}_{\chi\text{ mod } m}\chi(n)\overline{\chi}(r)=\phi(m)\delta_{n\equiv r\bmod m}-1,
\end{equation}
\begin{equation}\label{Gaussavg}
    \sideset{}{^*}{\sum}_{\chi\text{ mod } m}G(\chi)^2\overline{\chi}(r)=\phi(m)S(1,r;m) -1.
\end{equation}
Let $\chi, \psi$ be primitive characters with conductors $q_1,q_2$ respectively with $(q_1,q_2)=1$. Then we have
\begin{equation}
    G(\chi \psi)= G(\chi) G(\psi) \chi(q_2) \psi(q_1).
\end{equation}
\subsection{The set-up:}
Let $f, g$ be two modular forms of weight $k_1,k_2$ and level $N_1, N_2$ respectively. For a fixed prime $p$, we consider the sum
\begin{equation}
   \mbf{S}= \sum_{q_1\in \mc D}\sideset{}{^*}\sum_{\chi(q_1)}\sideset{}{^*}\sum_{\psi(q_2)}\overline{\chi}(p)\overline{\psi}(p)\overline{G(\chi)}^2L(1/2, f\times \chi)L(1/2, g\times \chi\psi).
\end{equation}
Using the approximate functional equation from \eqref{AFE} for the $L$-values $L(1/2, f\times \chi)$ and $L(1/2, g\times \chi\psi)$, we write
\begin{equation}
    \mbf{S}= \mbf S_{nn}+\mbf S_{nd}+\mbf S_{dn}+\mbf S_{dd},
\end{equation}
where the terms $\mbf S_{**}$ correspond to the dual and non-dual parts of the approximate functional equation and are given by-
\begin{align}
    \mbf S_{nn}&= \sum_{q_1\in \mc D}\sideset{}{^*}\sum_{\chi(q_1)}\sideset{}{^*}\sum_{\psi(q_2)} \sum_{m,n}\overline{G(\chi)}^2 \overline{\chi}(\bar{m}np)\overline{\psi}(np)\frac{\lambda_f(m)\lambda_g(n)}{(mn)^{1/2}}V_f\Big(\tfrac{mX}{q_\chi^{1/2}}\Big)V_g\Big(\tfrac{nY}{q_{\chi\psi}^{1/2}}\Big),\\ 
    \mbf S_{nd}&= \sum_{q_1\in \mc D}\sideset{}{^*}\sum_{\chi(q_1)}\sideset{}{^*}\sum_{\psi(q_2)} \sum_{m,n}\frac{q_1G(\psi)^2}{q_2} \overline{\chi}(\bar{m}np\bar{q_2}^2)\overline{\psi}(np\bar{q_1}^2)\frac{\lambda_f(m)\lambda_g(n)}{(mn)^{1/2}}V_f\left(\tfrac{mX}{q_\chi^{1/2}}\right)V_g\left(\tfrac{n}{Yq_{\chi\psi}^{1/2}}\right),\\
    \mbf S_{dn}&= \sum_{q_1\in \mc D}\sideset{}{^*}\sum_{\chi(q_1)}\sideset{}{^*}\sum_{\psi(q_2)} \sum_{m,n} \frac{|G(\chi)|^4}{q_1} \chi(\bar{m}n\bar{p}) \psi(n\bar{p})\frac{\lambda_f(m)\lambda_g(n)}{(mn)^{1/2}}V_f\Big(\tfrac{m}{Xq_\chi^{1/2}}\Big)V_g\Big(\tfrac{nY}{q_{\chi\psi}^{1/2}}\Big),\\
    \mbf S_{dd}&= \sum_{q_1\in \mc D}\sideset{}{^*}\sum_{\chi(q_1)}\sideset{}{^*}\sum_{\psi(q_2)} \sum_{m,n}\frac{G(\psi)^2G(\chi)^2}{q_2} \overline{\chi}(mnp\bar{q_2}^2)\overline{\psi}(np\bar{q_1}^2)\frac{\lambda_f(m)\lambda_g(n)}{(mn)^{1/2}}V_f\Big(\tfrac{m}{Xq_\chi^{1/2}}\Big)V_g\Big(\tfrac{n}{Yq_{\chi\psi}^{1/2}}\Big).
\end{align}

\section{Main Term}
\subsection{Evaluating \texorpdfstring{$\mbf S_{dn}$}{sdn}:} First, from the approximate functional equations, we consider the \textit{dual} part of $L(1/2, f\times \chi)$ and the \textit{non-dual}  part of $L(1/2, g\times \chi\psi)$. Let us denote its contribution to $\mbf S$ by  $\mbf S_{dn}$. A straight-forward calculation gives us
\begin{equation} \label{Sdn}
\begin{split}
    \mbf S_{dn}&= \sum_{q_1\in \mc D}  \sum_{m,n} q_1  \frac{\lambda_f(m)\lambda_g(n)}{(mn)^{1/2}} \Big(\sideset{}{^*}\sum_{\chi(q_1)}\chi(\bar{m}n\bar{p})\Big) \Big(\sideset{}{^*}\sum_{\psi(q_2)}\psi(n\bar{p})\Big) V_f\Big(\tfrac{m}{Xq_\chi^{1/2}}\Big)V_g\Big(\tfrac{nY}{q_{\chi\psi}^{1/2}}\Big). \\
\end{split}
\end{equation}
By evaluating the character sums using \eqref{charsum} and \eqref{Gaussavg}, we can write
\begin{equation}
    \mbf S_{dn}= U_1 -U_2-U_3+U_4,
\end{equation}
where $U_1,U_2,U_3,U_4$ are given by
\begin{equation} \label{U1}
    U_1 = \sum_{q_1\in \mc D} q_1 \phi(q_1 q_2) \underset{n \equiv mp \bmod{q_1}}{\underset{n\equiv p \bmod{q_2}}{\sum\nolimits_{m,n}}}   \frac{\lambda_f(m)\lambda_g(n)}{(mn)^{1/2}}  V_f\Big(\tfrac{m}{Xq_\chi^{1/2}}\Big)V_g\Big(\tfrac{nY}{q_{\chi\psi}^{1/2}}\Big),
\end{equation}
\begin{equation} \label{U2}
   U_2 = \sum_{q_1\in \mc D} q_1 \phi(q_1) \underset{n \equiv mp \bmod{q_1}}{\underset{n \not\equiv p \bmod{q_2}}{\sum\nolimits_{m,n}}}   \frac{\lambda_f(m)\lambda_g(n)}{(mn)^{1/2}}  V_f\Big(\tfrac{m}{Xq_\chi^{1/2}}\Big)V_g\Big(\tfrac{nY}{q_{\chi\psi}^{1/2}}\Big),
\end{equation}
\begin{equation} \label{U3}
   U_3 = \sum_{q_1\in \mc D} q_1 \phi(q_2) \underset{n \not\equiv mp \bmod{q_1}}{\underset{n \equiv p \bmod{q_2}}{\sum\nolimits_{m,n}}}   \frac{\lambda_f(m)\lambda_g(n)}{(mn)^{1/2}}  V_f\Big(\tfrac{m}{Xq_\chi^{1/2}}\Big)V_g\Big(\tfrac{nY}{q_{\chi\psi}^{1/2}}\Big),
\end{equation}
\begin{equation} \label{U4}
   U_4 = \sum_{q_1\in \mc D} q_1 \underset{n \not\equiv mp \bmod{q_1}}{\underset{n \not\equiv p \bmod{q_2}}{\sum\nolimits_{m,n}}}   \frac{\lambda_f(m)\lambda_g(n)}{(mn)^{1/2}}  V_f\Big(\tfrac{m}{Xq_\chi^{1/2}}\Big)V_g\Big(\tfrac{nY}{q_{\chi\psi}^{1/2}}\Big).
\end{equation}
Using the decay of $V_f$ and $V_g$ from \eqref{Vbound}, the sums over $m$ and $n$ can be truncated at $\ll_{k_1, N_1} q_1^{1+\epsilon}X$ and $\ll_{k_2, N_2}  (q_1q_2)^{1+\epsilon}/Y$ respectively with negligible error in $Q_1$ and $q_2$. Thus for the rest of the calculation we restrict to $m\ll_{k_1, N_1} q_1^{1+\epsilon}X$ and $n\ll_{k_2, N_2}  (q_1q_2)^{1+\epsilon}/Y$.

\textit{Calculation of $U_{1}$}: We note that
\begin{equation*}
  \{(m,n)| n\equiv p \bmod{q_2}, n \equiv mp \bmod{q_1}, m \leq q_1^{1+\epsilon}X, n \leq (q_1q_2)^{1+\epsilon}/Y \} = A_1 \cup A_2 \cup A_3  ,
\end{equation*}
    where
\begin{align}
  A_1 &=\{ (m,n)|m=1,n=p\}, \\
    A_2 &= \{ (m,n)|m >q_1,n=p, m \equiv 1 (\bmod q_1), m \leq q_1^{1+\epsilon}X \}, \nonumber \\
     A_3 &= \{ (m,n)|n >q_2, n\equiv p \bmod{q_2}, n \equiv mp \bmod{q_1}, m \leq q_1^{1+\epsilon}X, n \leq (q_1q_2)^{1+\epsilon}/Y \} \nonumber.
\end{align}
Let the summand of $U_1$ under the three disjoint sets $A_1, A_2, A_3 $ be $U_{1,1}, U_{1,2}, U_{1,3}$ respectively. Using \eqref{Vasymp} and the fact that $q_{\chi} \asymp_{k_1,N_1} q_1^{2}, q_{\chi\psi} \asymp_{k_2,N_2} (q_1q_2)^{2}$,  we calculate
\begin{equation} \label{U11}
\begin{split}
     U_{1,1}&= \frac{\lambda_{g}(p)}{p^{1/2}}\sum_{q_1\in \mc D}q_1\phi(q_1q_2) + O( Q_1^{11/4} q_2 X^{-1/4}) + O(Q_1^{11/4} q_2^{3/4} Y^{1/4}).\\
     &= \frac{\lambda_{g}(p)}{p^{1/2}}\pi_2(Q_1) q_2 +  O( Q_1^{11/4} q_2 X^{-1/4}) + O(Q_1^{11/4} q_2^{3/4} Y^{1/4}).
\end{split}
\end{equation}
For $U_{1,2}$, the number of $m \leq q_1^{1+\epsilon} X$, satisfying $m \equiv 1 \bmod{q_1}$ is $O(q_1^{\epsilon} X)$. Thus using \eqref{Vbound} we calculate 
\begin{align} \label{U12}
&U_{1,2} = \sum_{q_1\in \mc D} q_1 \phi(q_1 q_2) \underset{m \equiv 1 \bmod{q_1}}{\underset{m \leq q_1^{1+\epsilon}X} {\sum_{m >q_1}} } \frac{\lambda_f(m)\lambda_g(p)}{(mp)^{1/2}} V_f\Big(\tfrac{m}{Xq_\chi^{\frac{1}{2}}}\Big)V_g\Big(\tfrac{pY}{q_{\chi\psi}^{\frac{1}{2}}}\Big) \\
 &\ll \sum_{q_1\in \mc D} q_1^2 q_2 \frac{X}{q_1^{1/2-\epsilon}} \ll Q_1^{5/2+\epsilon}q_{2} X. \nonumber
\end{align}
For $U_{1,3}$, for a fixed $q_1$ and $p$, we put $n= p+rq_2$, where $r$ varies over the integers in the range $[ 1, q_1^{1+\epsilon} q_2^{\epsilon}/Y ]$. We plug this in the congruence relation $n \equiv mp \bmod{q_1}$ to get $p+rq_2 \equiv mp \bmod{q_1}$. We observe that for each $m \leq q_1^{1+\epsilon} X$, there is a unique $r$  modulo $q_1$ that satisfies the above congruence relation as $q_2$ is prime. Thus for each $m \leq q_1^{1+\epsilon} X$ the numbers of possible $r$ are at most $(q_1q_2)^{\epsilon}/Y$. Hence, we write
\begin{align} \label{u13}
    U_{1,3} &\ll \sum_{q_1\in \mc D} q_1^2 q_2 \underset{n \equiv p \bmod{q_2}}{\underset{ q_2 < n \leq (q_1 q_2)^{1+\epsilon}/Y, } {\sum_{m \leq q_1^{1+\epsilon}X, n \equiv mp \bmod{q_1}}} } \frac{|\lambda_f(m)\lambda_g(n)|}{(mn)^{1/2}} \nonumber \\ 
    &\ll  \sum_{q_1\in \mc D} q_1^2 q_2 \sum_{m \ll q_1^{1+\epsilon}X} \frac{|\lambda_{f}(m)|}{m^{1/2}} \underset{ p+rq_2 \equiv mp \bmod{q_1}} {\sum_{1 \leq r \leq q_1^{1+\epsilon}q_2^{\epsilon}/Y}}  \frac{|\lambda_g(p+rq_2)|}{(p+rq_2)^{1/2}} \ll Q_1^{\tfrac{7}{2}+\epsilon } q_2^{\tfrac{1}{2} +\epsilon} X^{\tfrac{1}{2}}Y^{-1} .
    \end{align}
In the last summation we have used Deligne's bound for the numerator and use the fact that each term is bounded by  $q_2^{-1/2+\epsilon}$. Then we applied the fact that there are only $(q_1q_2)^{\epsilon}/Y$ possibilities of $r$.

\textit{Estimation of $U_2$}: From \eqref{U2} and \eqref{Vbound} we note that
\begin{align} \label{cal U2}
  U_2 &\ll  \sum_{q_1\in \mc D} q_1^2 \sum_{m \leq q_1^{1+\epsilon} X}\Big(\underset{n \equiv mp \bmod{q_1}}{\sum_{n \leq q_1^{1+\epsilon}X}}\frac{|\lambda_f(m)\lambda_g(n)|}{(mn)^{1/2}}+ \underset{n \equiv mp \bmod{q_1}}{\sum_{q_1^{1+\epsilon} X< n \leq (q_1q_2)^{1+\epsilon}/Y}}\frac{|\lambda_f(m)\lambda_g(n)|}{(mn)^{1/2}}\Big)  \\
   & \ll Q_1^{4+\epsilon} X+ \sum_{n \ll (Q_1 q_2)^{1+\epsilon}/Y}\frac{|\lambda_{g}(n)|}{n^{1/2}} \sum_{m \ll Q_1^{1+\epsilon}X} \frac{|\lambda_{f}(m)|}{m^{1/2}} \underset{q_1 \ll Q_1}{\sum_{q_1| (n-mp)}} q_1^2 \nonumber \\
   &\ll Q_1^{4+\epsilon}X+Q_1^{3+\epsilon}q_2^{1/2+\epsilon}X^{1/2}Y^{-1/2} \nonumber.
\end{align}
\textit{Estimation of $U_3$}: Using \eqref{U3} and \eqref{Vbound} we write that
\begin{align} \label{cal U3}
     &U_3 \ll \sum_{q_1\in \mc D} q_1q_2 \underset{n= p}{\sum_{m \leq q_1^{1+\epsilon} X}}   \frac{|\lambda_f(m)\lambda_g(n)|}{(mn)^{1/2}} + \sum_{q_1\in \mc D} q_1q_2 \underset{n \equiv p \bmod{q_2}}{\underset{q_2< n \leq (q_1q_2)^{1+\epsilon}/Y}{\sum_{m \leq q_1^{1+\epsilon} X}}}   \frac{|\lambda_f(m)\lambda_g(n)|}{(mn)^{1/2}}  \\
     &\ll Q_1^{\frac{5}{2}+\epsilon}q_2 X^{\frac{1}{2}} + \sum_{q_1\in \mc D} q_1 q_2 \sum_{m \ll q_1^{1+\epsilon}X} \frac{|\lambda_{f}(m)|}{m^{1/2}}  \sum_{1 \leq r \leq q_1^{1+\epsilon}q_2^{\epsilon}/Y}  \frac{|\lambda_g(p+rq_2)|}{(p+rq_2)^{1/2}} \nonumber \\
     &\ll Q_1^{5/2+\epsilon}q_2 X^{1/2}+ Q_1^{3+\epsilon}q_2^{1/2+\epsilon}X^{1/2}Y^{-1/2} \nonumber.
\end{align}
\textit{Estimation of $U_4$}: From \eqref{U4} and \eqref{Vbound} we get that
\begin{equation} \label{cal U4}
    U_4 \ll \sum_{q_1\in \mc D} q_1 \underset{n \leq (q_1q_2)^{1+\epsilon}/Y }{\sum_{m \leq q_1^{1+\epsilon} X}}   \frac{|\lambda_f(m)\lambda_g(n)|}{(mn)^{1/2}} \ll Q_{1}^{3+\epsilon}q_2^{1/2+\epsilon}X^{1/2}Y^{-1/2} .
\end{equation}
Thus combining all the above cases \eqref{U11}, \eqref{U12}, \eqref{u13} \eqref{cal U2}, \eqref{cal U3}, \eqref{cal U4} we get that 
\begin{align} \label{finalSdn}
\mbf S_{dn} &= \frac{\lambda_{g}(p)}{p^{1/2}}\pi_2(Q_1) q_2 + O(Q_1^{5/2 +\epsilon} q_2 X + Q_1^{5/2+\epsilon} q_2 X^{1/2} + Q_1^{11/4} q_2 X^{-1/4}) \\ 
& +O( Q_1^{4+\epsilon}X +Q_1^{11/4} q_2^{3/4} Y^{1/4}  +Q_1^{3+\epsilon}q_2^{1/2+\epsilon}X^{1/2}Y^{-1/2} + Q_1^{\tfrac{7}{2}+\epsilon } q_2^{\tfrac{1}{2} +\epsilon} X^{\tfrac{1}{2}}Y^{-1} ). \nonumber
\end{align}

\section{Error Terms}
\subsection{Estimating \texorpdfstring{$\mbf S_{dd}$}{sdd}:}\label{sec:sdd}
Consider the \textit{dual} parts from  the approximate functional equations of $L(1/2, f\times \chi)$ and $L(1/2, g\times \chi\psi)$ and we denote its contribution to $\mbf S$ by  $\mbf S_{dd}$. Then using \eqref{Gaussavg}, we see that
\begin{align}
    \mbf S_{dd}&= \sum_{q_1\in \mc D}\sideset{}{^*}\sum_{\chi(q_1)}\sideset{}{^*}\sum_{\psi(q_2)} \sum_{m,n}\frac{G(\psi)^2G(\chi)^2}{q_2} \overline{\chi}(mnp\bar{q_2}^2)\overline{\psi}(np\bar{q_1}^2)\frac{\lambda_f(m)\lambda_g(n)}{(mn)^{1/2}}V_f\Big(\tfrac{m}{Xq_\chi^{1/2}}\Big)V_g\Big(\tfrac{n}{Yq_{\chi\psi}^{1/2}}\Big)\nonumber\\
    &= \sum_{q_1\in \mc D}\sum_{m,n}\frac{\phi(q_1q_2)}{q_2}\frac{\lambda_f(m)\lambda_g(n)}{(mn)^{1/2}} S(1,np\bar{q_1}^2;q_2)S(1,mnp\bar{q_2}^2;q_1)V_f\Big(\tfrac{m}{Xq_\chi^{1/2}}\Big)V_g\Big(\tfrac{n}{Yq_{\chi\psi}^{1/2}}\Big)\\
    &\quad\quad+O(Q_1^{2+\epsilon}q_2^{1+\epsilon}\sqrt{XY})\nonumber.
\end{align}
Trivially, using the Weil bound for Kloosterman sum, we see that $\mbf S_{dd} \ll Q_1^{7/2+\epsilon}q_2^{1+\epsilon} \sqrt{X} \sqrt{Y}.$

The sums over $m$ and $n$ can be truncated at $\ll_{k_1, N_1} M:= Q_1^{1+\epsilon}X$ and $\ll_{k_2, N_2} N:= (Q_1q_2)^{1+\epsilon}Y$ respectively with negligible error in $Q_1$ and $q_2$ (from \eqref{Vbound}). Using the Mellin inversions for $V_f$ and $V_g$ we can write
\begin{equation}\label{Sddint}
\begin{split}
     \mbf S_{dd}=\frac{1}{(2\pi i)^2}\int_{(\epsilon)}\int_{(\epsilon)}&\sum_{q_1\in \mc D}\sum_{m\ll M}\sum_{n\ll N}\frac{\phi(q_1q_2)}{q_1^{-u/2-v/2}q_2^{1-v/2}}\frac{\lambda_f(m)\lambda_g(n)}{(mn)^{1/2}m^un^v} S(1,np\bar{q_1}^2,q_2)S(1,mnp\bar{q_2}^2,q_1)\\
     &\times \frac{\gamma_f(\tfrac{1}{2}+u)}{\gamma_f(\tfrac{1}{2})}\frac{\gamma_g(\tfrac{1}{2}+v)}{\gamma_g(\tfrac{1}{2})}X^u Y^v\frac{du}{u}\frac{dv}{v}+O(Q_1^{2+\epsilon}q_2^{1+\epsilon}\sqrt{XY}).
\end{split}
\end{equation}
We note that the ratio of gamma functions in \eqref{Sddint} is bounded by $\exp(-(Q_1q_2)^\epsilon)$ when $|\Im(u)|> (Q_1q_2)^\epsilon$ and $|\Im(v)|>(Q_1q_2)^\epsilon$. Thus it is enough to evaluate the integrals in the range $-(Q_1q_2)^\epsilon\ll |\Im(u)|,|\Im(v)| \ll (Q_1q_2)^\epsilon $.

Now consider the sums over $q_1$ and $m,n$ inside the integral and let
\begin{equation}
    \mbf T(u,v)=\sum_{q_1\in \mc D}\sum_{m\ll M}\sum_{n\ll N}\frac{\phi(q_1q_2)}{q_1^{-u/2-v/2}q_2^{1-v/2}}\frac{\lambda_f(m)\lambda_g(n)}{(mn)^{1/2}m^un^v} S(1,np\bar{q_1}^2;q_2)S(1,mnp\bar{q_2}^2;q_1).
\end{equation} 
Taking the absolute values, we get
\begin{equation}
    \mbf T(u,v)\ll q_2^{\epsilon}\sum_{m\ll M}\sum_{n\ll N} \frac{|\lambda_f(m)\lambda_g(n)|}{(mn)^{1/2}}\left|\sum_{q_1} \frac{\phi(q_1)}{q_1^{-u/2-v/2}}S(1,np\bar{q_1}^2;q_2)S(1,mnp\bar{q_2}^2;q_1)\right|.
\end{equation}
Now using the Cauchy-Schwarz inequalities for the $q_2, m$ and $n$ sums, we get
\begin{equation}
    \mbf T(u,v)\ll q_2^{\epsilon}Q_1^{\epsilon} \sqrt{\mc T(u,v)},
\end{equation}
where
\begin{equation}
    \mc T(u,v)=\sum_{n\ll N}\sum_{m\ll M}\left|\sum_{q_1} \frac{\phi(q_1)}{q_1^{-u/2-v/2}}S(1,np\bar{q_1}^2;q_2)S(1,mnp\bar{q_2}^2;q_1)\right|^2.
\end{equation}
Thus we have
\begin{equation}\label{Sdd}
    \mbf S_{dd}\ll q_2^{\epsilon}Q_1^{\epsilon} 
    \int_{- (Q_1q_2)^{\epsilon}}^{ (Q_1q_2)^{\epsilon}}\int_{- (Q_1q_2)^{\epsilon}}^{ (Q_1q_2)^{\epsilon}} \sqrt{\mc T(\epsilon+it,\epsilon+it')} dt dt'+ Q_1^{2+\epsilon}q_2^{1+\epsilon}\sqrt{XY}.
\end{equation}
\subsubsection{Evaluating $\mc T(u,v)$}
First, we introduce a smooth, compactly supported function $W(x)$ with bounded derivatives for the sum over $n$ to get
\begin{equation}
    \mc T(u,v)\le \sum_{m\ll M}\sum_{n\in \mbb Z}\left|\sum_{q_1} \frac{\phi(q_1)}{q_1^{-u/2-v/2}}S(1,np\bar{q_1}^2;q_2)S(1,mnp\bar{q_2}^2;q_1)\right|^2 W(n/N).
\end{equation}
Expanding the square, we find that
\begin{equation}\label{TUV}
    \mc T(u,v)\ll Q_1^{2+\epsilon}\sum_{m\ll M}\sum_{q_1}\sum_{q_1'}|\mc T_{m, q_1,q_2,q_2'}|,
\end{equation}
where
\begin{equation}
    \mc T_{m, q_1,q_1',q_2}:=\sum_{n\in \mbb Z}S(1,np\bar{q_1}^2;q_2)S(1,mnp\bar{q_2}^2;q_1)S(1,np\bar{q_1'}^2;q_2)S(1,mnp\bar{q_2}^2;q_1')W(n/N).
\end{equation}
Splitting the sum over $n$ into congruences modulo $q_1q_1'q_2$, we get 
\begin{equation}
\begin{split}
    \mc T_{m, q_1,q_1',q_2}=&\sum_{\alpha\bmod q_1q_1'q_2} S(1,\alpha p\bar{q_1}^2;q_2)S(1,m\alpha p\bar{q_2}^2;q_1)S(1,\alpha p\bar{q_1'}^2;q_2)S(1,m\alpha p\bar{q_2}^2;q_1')\\
    &\times \sum_{n} W\left(\frac{\alpha+nq_1q_1'q_2}{N}\right).
\end{split}
\end{equation}
Using the Poisson summation formula for the sum over $n$ we get
\begin{equation}
\mc 
T_{m,q_1,q_1',q_2}=\frac{N}{q_1q_1'q_2}\sum_{n\in\mbb{Z}}\mc 
T_1(n;m,q_1,q_1',q_2)\int_{\mbb
 R}W(y)e\left(-\frac{nNy}{q_1q_1'q_2}\right)dy,
\end{equation}
where $\mc T(n;m,q_1,q_1',q_2)$ is given by
\begin{equation}
   \sum_{\alpha\bmod q_1q_1'q_2} S(1,\alpha p\bar{q_1}^2;q_2)S(1,m\alpha p\bar{q_2}^2;q_1)S(1,\alpha p\bar{q_1'}^2;q_2)S(1,m\alpha p\bar{q_2}^2;q_1') e_{q_1q_1'q_2}(\alpha n).
\end{equation}
Integrating by parts, we see that the integral is negligibly small for $|n|\gg (Q_1^2q_2)^{1+\epsilon}/N$. Thus
\begin{equation}\label{aftrPS}
    \mc T_{m, q_1,q_1',q_2}\ll \frac{q_2^{\epsilon}Y}{Q_1^{1-\epsilon}}\sum_{|n|\ll Q_1^{1+\epsilon}/Y}|\mc T(n;m,q_1,q_1',q_2)|+(Q_1q_2)^{-100}.
\end{equation}
Since $(q_2,q_1q_1')=1$, $\alpha$ can be uniquely written as
\begin{equation}
\alpha=\alpha_1q_2\bar{q_2}+\alpha_2q_1q_1'\bar{q_1}\bar{q_1'},
\end{equation}
where $\alpha_1$ varies modulo $q_1q_1'$ and $\alpha_2$ over $q_2$. Thus the sum $\mc T(n;m,q_1,q_2,q_2')$ splits into the product of
\begin{equation}
\mfk T^2(n;q_1,q_1',q_2)=\sum_{\alpha_2\md q_2}e_{q_2}(\alpha_2n \bar{q_1}\bar{q_1}')S(1,\alpha_2 p\bar{q_1}^2;q_2)S(1,\alpha_2 p\bar{q_1'}^2;q_2)
\end{equation}
and
\begin{equation}
\mfk T^1(n;q_1,q_1',q_2)=\sum_{\alpha_1\md q_1q_1'}e_{q_1q_1'}(\alpha_1n \bar{q_2})S(1,m\alpha_1 p\bar{q_2}^2;q_1)S(1,m\alpha_1 p\bar{q_2}^2;q_1').
\end{equation}

\subsubsection*{Bound for $\mfk T^1(n;q_1,q_1',q_2)$:} When $q_1=q_1'$, we observe that by trivially bounding the terms we get
\begin{equation}
    \mfk T^1(n;q_1,q_1',q_2)\ll Q_1^3.
\end{equation}
For the case when $q_1\neq q_1'$, we again split the sum over $\alpha_1$ and get
\begin{equation} \label{T^1}
    \begin{split}
         \mfk T^1(n;q_1,q_1',q_2)= \sum_{\alpha_3 \md q_1}e_{q_1}(\alpha_3n \bar{q_2})S(1,m\alpha_3 p\bar{q_2}^2;q_1) \times \sum_{\alpha_4 \md q_1'}e_{q_1'}(\alpha_4n \bar{q_2})S(1,m\alpha_4 p\bar{q_2}^2;q_1').
    \end{split}
\end{equation}
Let us consider the first sum on the right-hand side. We open the Kloosterman sum and get
\begin{equation}
    \sideset{}{^*}\sum_{a\md q_1}e_{q_1}(a)\sum_{\alpha_3\md q_1}e_{q_1}(\alpha_3(n\bar{q_2}+\bar{a}mp\bar{q_2}^2)).
\end{equation}
We note that the inner sum vanishes if $q_1|n$ and when $(q_1,n)=1$, the whole sum is $\ll q_1$.  The same analysis also holds for the second term on the right-hand side of \eqref{T^1}. Thus we conclude that 
\begin{equation}
  \mfk T^1(n;q_1,q_1',q_2)\ll  \begin{cases}
    Q_1^3 & \text{ if } q_1=q_1';\\
    Q_1^2 & \text{ if } n\neq 0 \text{ and } q_1\neq q_1';\\
    0 & \text{ if } n= 0 \text{ and } q_1\neq q_1'.
    \end{cases}
\end{equation}
\subsubsection*{Bound for $\mfk T^2(n;q_1,q_1',q_2)$:} We open the Kloosterman sums in $\mfk T^1(n;q_1,q_1',q_2)$ and get
\begin{equation}\label{T2cong}
\begin{split}
    \sideset{}{^*}\sum_{a,b\md q_2}e_{q_2}(a+b)\sum_{\alpha_2\md q_2}e_{q_2}(\alpha_2(nq_1q_1'+\bar{a}p\bar{q_1}^2+\bar{b}p\bar{q_1'}^2))= q_2\underset{nq_1q_1'+\bar{a}p\bar{q_1}^2+\bar{b}p\bar{q_1'}^2\equiv 0\md q_2}{\sideset{}{^*}\sum_{a,b\md q_2}e_{q_2}(a+b)}.
\end{split}
\end{equation}
We observe that $\bar a$ can be uniquely solved as $\bar{a}\equiv q_1^2 \bar{p}\bar{b}\bar{q_1'}^2(p-nbq_1q_1'^3)\md q_2 $. Let us write $j=p-nbq_1q_1'^3$ and  we see that $(j,q_2)=1$.

First, let us consider the case when $(q_2,n)=1$. In this case, we have $(j(p-g),q_2)=1$. From which we get
\begin{equation}
    \mfk T^2(n;q_1,q_1',q_2)=q_2 \underset{(p-j,q_2)=1}{\sideset{}{^*}\sum_{j\md q_2}}e_{q_2}(\bar{n}\bar{q_1}\bar{q_1'}^3(p-j)+pq_1'^2 \bar{n}\bar{q_1'}(p-j)\bar{j}).
\end{equation}
By adding and subtracting the missing $j=p$ term, we get
\begin{equation}
    \mfk T^2(n;q_1,q_1',q_2)=q_2 \sideset{}{^*}\sum_{j\md q_2}e_{q_2}(\bar{n}\bar{q_1}\bar{q_1'}^3(p-j)+pq_1'^2 \bar{n}\bar{q_1'}(p-j)\bar{j})-1.
\end{equation}
The sum over $j$ is now $\ll |S(-\bar{n}\bar{q_1}\bar{q_1'}^3,p^2q_1'^2 \bar{n}\bar{q_1'}; q_2)|\ll q_2^{1/2} $. Thus when $(q_2,n)=1$ we get 
\begin{equation}
    \mfk T^2(n;q_1,q_1',q_2)\ll q_2^{3/2}.
\end{equation}
When $q_2|n$, the congruence in \eqref{T2cong} reduces to $\bar{a}p\bar{q_1}^2+\bar{b}p\bar{q_1'}^2\equiv 0\md q_2$. Thus
\begin{equation}
 \mfk T^2(n;q_1,q_1',q_2)  = q_2\sideset{}{^*}\sum_{a\md q_2}e_{q_2}(a-q_1^2\bar{q_1'}^2a).
\end{equation}
The inside sum is clearly $\phi(q_2)$ if $q_2|(q_1^2-q_1'^2)$ and $-1$, otherwise.  Now, by requiring that $Q_1^{1+\delta}\ll q_2$, it becomes evident that $q_2|q_1^2-q_1'^2$ if and only if $q_2|q_1-q_1'$. Hence, we obtain
\begin{equation}
  \mfk T^2(n;q_1,q_1',q_2)\ll  \begin{cases}
    q_2 (q_2, q_1-q_1') & \text{ if } q_2|n;\\
    q_2^{3/2} & \text{ otherwise}.
    \end{cases}
\end{equation}
Therefore, we can deduce that
\begin{equation}
    \mc T(n;m,q_1,q_1',q_2)\ll \begin{cases}
        0 & \text{ if } n=0 \text{ and } q_1\neq q_1';\\
        Q_1^3 q_2^2 & \text{ if } n=0 \text{ and } q_1= q_1';\\
        Q_1^2 q_2^{3/2} (q_2,n)^{-1} & \text{ if } n\neq 0 \text{ and } q_1\neq q_1';\\
        Q_1^3 q_2^{3/2} (q_2,n)^{1/2} & \text{ if } n\neq 0 \text{ and } q_1= q_1'.
    \end{cases}
\end{equation}
Substituting the above bound for $\mc T(n;m,q_1,q_1',q_2)$ back into \eqref{aftrPS} gives us the following bounds:
\begin{equation}
    \mc T_{m, q_1,q_1',q_2}\ll \begin{cases}
        Q_1^{2+\epsilon}q_2^{2+\epsilon}Y+Q_1^{3+\epsilon} q_2^{3/2+\epsilon} & \text{ if } q_1=q_1';\\
        Q_1^{2+\epsilon} q_2^{3/2+\epsilon} & \text{ otherwise }.
    \end{cases}
\end{equation}
As a consequence, we get from \eqref{TUV} that
\begin{equation}
    \mc T(u,v) \ll (Q_1^{6+\epsilon} q_2^{2+\epsilon} X Y + Q_1^{7+\epsilon} q_2^{3/2+\epsilon} X).
\end{equation}
Finally, putting everything together in \eqref{Sdd}, we get the following estimate that
\begin{equation} \label{finalSdd}
    \mbf S_{dd}\ll Q_1^{3+\epsilon} q_2^{1+\epsilon} X^{1/2} Y^{1/2} + Q_1^{7/2+\epsilon} q_2^{3/4+\epsilon} X^{1/2}.
\end{equation}
\subsection{Estimating \texorpdfstring{$\mbf S_{nd}$}{snd}:}\label{sec:snd}

Consider the \textit{non-dual} part from  the approximate functional equations of $L(1/2, f\times \chi)$ and the \textit{dual} part from $L(1/2, g\times \chi\psi)$ and we denote its contribution to $\mbf S$ by  $\mbf S_{nd}$. Then we see that
\begin{equation}\label{Sndmain}
\begin{split}
    \mbf S_{nd}&= \sum_{q_1\in \mc D}\sideset{}{^*}\sum_{\chi(q_1)}\sideset{}{^*}\sum_{\psi(q_2)} \sum_{m,n}\frac{q_1G(\psi)^2}{q_2} \overline{\chi}(\bar{m}np\bar{q_2}^2)\overline{\psi}(np\bar{q_1}^2)\frac{\lambda_f(m)\lambda_g(n)}{(mn)^{1/2}}V_f\left(\tfrac{mX}{q_\chi^{1/2}}\right)V_g\left(\tfrac{n}{Yq_{\chi\psi}^{1/2}}\right)\\
    &=\sum_{q_1\in \mc D}\underset{np\equiv mq_2^2\bmod q_1}{\sum_{m,n}}\frac{q_1\phi(q_1q_2)}{q_2}\frac{\lambda_f(m)\lambda_g(n)}{(mn)^{1/2}} S(1,np\bar{q_1}^2;q_2) \\
    & \qquad \qquad \qquad \times V_f\left(\tfrac{mX}{q_\chi^{1/2}}\right)V_g\left(\tfrac{n}{Yq_{\chi\psi}^{1/2}}\right)+O\Big(\frac{Q_1^{3+\epsilon} q_2^{1+\epsilon}Y^{\frac{1}{2}}}{X^{\frac{1}{2}}}\Big)  .
\end{split}
\end{equation}
To get the error term in \eqref{Sndmain} we first truncate the sums over $m$ and $n$  at $\ll_{k_1, N_1} q_1^{1+\epsilon}/X$ and $\ll_{k_2, N_2}  (q_1q_2)^{1+\epsilon}Y$ (respectively, with negligible error). Next, we evaluate the character sums using \eqref{charsum}, \eqref{Gaussavg} and applying Weil bound on the Kloosterman sums.

The summation over $q_1,q_2, m,n$ in \eqref{Sndmain} can be rewritten as
\begin{equation}
   \sum_{m\ll Q_1^{1+\epsilon}/X}\sum_{n\ll (Q_1q_2)^{1+\epsilon}Y}\frac{\lambda_f(m)\lambda_g(n)}{(mn)^{1/2}} \sum_{q_1|(np- mq_2^2)}\frac{q_1\phi(q_1q_2)}{q_2} S(1,np\bar{q_1}^2;q_2)V_f\left(\tfrac{mX}{q_\chi^{1/2}}\right)V_g\left(\tfrac{n}{Yq_{\chi\psi}^{1/2}}\right).
\end{equation}
First, we observe that there is no contribution from the \textit{zero-frequency} term $np=mq_2^2$. To see this, since $n\ll (Q_1q_2)^{1+\epsilon}Y$ and $p\ll \log Q_1$, $np\ll (Q_1q_2)^{1+\epsilon}Y$. Thus by requiring that $Q_1\ll q_2$, we see that $np\ll q_2^{2+\epsilon}Y< mq_2^2$, by choosing $Y<q_2^{-\delta}$ for some small $\delta>0$.

Now, we note that the number of $q_1|(np-mq_2^2$) is $\ll q_2^\epsilon$. Thus we see that
\begin{equation}
    \sum_{q_1|(np- mq_2^2)}\frac{q_1\phi(q_1q_2)}{q_2} \left|S(1,np\bar{q_1}^2;q_2) V_f\Big(\tfrac{mX}{q_\chi^{1/2}}\Big) V_g\Big(\tfrac{n}{Yq_{\chi\psi}^{1/2}}\Big)\right|\ll Q_1^2q_2^{1/2+\epsilon}.
\end{equation}
Thus we get
\begin{equation}
    \mbf S_{nd}\ll Q_1^2q_2^{1/2+\epsilon}\sum_{m\ll Q_1^{1+\epsilon}/X}\sum_{n\ll (Q_1q_2)^{1+\epsilon}Y}\frac{|\lambda_f(m)\lambda_g(n)|}{(mn)^{1/2}}.
\end{equation}
Using Cauchy-Schwarz to evaluate the $m$ and $n$ sums, we get
\begin{equation} \label{finalSnd}
     \mbf S_{nd}\ll Q_1^2q_2^{1/2+\epsilon} \cdot Q_1^{1/2+\epsilon}X^{-1/2}\cdot (Q_1q_2)^{1/2+\epsilon} Y^{1/2}\ll Q_1^{3+\epsilon} q_2^{1+\epsilon} X^{-1/2} Y^{1/2}.
\end{equation}

\subsection{Estimating \texorpdfstring{$\mbf S_{nn}$}{snn}:}
Consider the \textit{non-dual} parts from  the approximate functional equations of $L(1/2, f\times \chi)$ and $L(1/2, g\times \chi\psi)$ and we denote its contribution to $\mbf S$ by $\mbf S_{nn}$. Then we see that
\begin{equation}\label{Snnmaineq}
\begin{split}
    \mbf S_{nn}&= \sum_{q_1\in \mc D}\sideset{}{^*}\sum_{\chi(q_1)}\sideset{}{^*}\sum_{\psi(q_2)} \sum_{m,n}\overline{G(\chi)}^2 \overline{\chi}(\bar{m}np)\overline{\psi}(np)\frac{\lambda_f(m)\lambda_g(n)}{(mn)^{1/2}}V_f\Big(\tfrac{mX}{q_\chi^{1/2}}\Big)V_g\Big(\tfrac{nY}{q_{\chi\psi}^{1/2}}\Big) \\
    &=\sum_{q_1\in \mc D}\underset{n\equiv p\bmod q_2}{\sum_{m,n}}\phi(q_1q_2)\frac{\lambda_f(m)\lambda_g(n)}{(mn)^{1/2}} S(1,m\bar{n}\bar{p};q_1)V_f\Big(\tfrac{mX}{q_\chi^{1/2}}\Big)V_g\Big(\tfrac{nY}{q_{\chi\psi}^{1/2}}\Big) \\
    & -\sum_{q_1\in \mc D}\underset{n\equiv p\bmod q_2}{\sum_{m,n}}\phi(q_2)\frac{\lambda_f(m)\lambda_g(n)}{(mn)^{1/2}} V_f\Big(\tfrac{mX}{q_\chi^{1/2}}\Big)V_g\Big(\tfrac{nY}{q_{\chi\psi}^{1/2}}\Big) \\
    &-\sum_{q_1\in \mc D} \sum_{m,n} \phi(q_1)\frac{\lambda_f(m)\lambda_g(n)}{(mn)^{1/2}} S(1,m\bar{n}\bar{p};q_1)V_f\Big(\tfrac{mX}{q_\chi^{1/2}}\Big)V_g\Big(\tfrac{nY}{q_{\chi\psi}^{1/2}}\Big) + O\Big(\frac{Q_1^{2+\epsilon} q_2^{1/2+\epsilon}}{(XY)^{1/2}}\Big).
\end{split}  
\end{equation}
The error term comes from truncating the sum over $m$ and $n$ at $\ll_{k_1, N_1} q_1^{1+\epsilon}/X$ and $\ll_{k_2, N_2}  (q_1q_2)^{1+\epsilon}/Y$ respectively with negligible error. After this truncation, let the three main terms in \eqref{Snnmaineq} be $T_1, T_2$ and $T_3$ respectively.
Thus, using Weil's bound for Kloosterman sums and \eqref{Vbound} we get
\begin{align} 
 T_1 &\ll \sum_{q_1\in \mc D}q_1q_2 \underset{n= p}{\sum_{m \leq q_1^{1+\epsilon}/X}}\frac{|\lambda_f(m)\lambda_g(n)|}{(mn)^{1/2}} q_1^{\frac{1}{2}} +\sum_{q_1\in \mc D} q_1q_2 \underset{n\equiv p\bmod q_2}{\underset{q_2<n\leq (q_1q_2)^{1+\epsilon}/Y}{\sum_{m \leq q_1^{1+\epsilon}/ X }}}\frac{|\lambda_f(m)\lambda_g(n)|}{(mn)^{1/2}} q_{1}^{\frac{1}{2}}\nonumber \\
    &\ll Q_1^{3+\epsilon}q_2 X^{-1/2} + \sum_{q_1\in \mc D} q_1^{3/2} q_2 \sum_{m \ll q_1^{1+\epsilon}/X} \frac{|\lambda_{f}(m)|}{m^{1/2}}  \sum_{1 \leq r \leq q_1^{1+\epsilon}q_2^{\epsilon}/Y}  \frac{|\lambda_g(p+rq_2)|}{(p+rq_2)^{1/2}} \nonumber  \nonumber \\
    &\ll Q_1^{3+\epsilon}q_2 X^{-1/2}+Q_1^{7/2+\epsilon} q_2^{1/2+\epsilon} X^{-1/2} Y^{-1/2}.
\end{align}
For $T_2$ we calculate
\begin{align}
    T_2 &\ll  \sum_{q_1\in \mc D} \phi(q_2) \underset{n= p}{\sum_{m \leq q_1^{1+\epsilon}/X}}\frac{|\lambda_f(m)\lambda_g(n)|}{(mn)^{1/2}} +\sum_{q_1\in \mc D} \phi(q_2) \underset{n\equiv p\bmod q_2}{\underset{q_2<n\leq (q_1q_2)^{1+\epsilon}/Y}{\sum_{m \leq q_1^{1+\epsilon}/ X }}}\frac{|\lambda_f(m)\lambda_g(n)|}{(mn)^{1/2}} \\
    & \ll Q_1^{3/2 +\epsilon} q_2 X^{-1/2} + Q_1^{2+\epsilon} q_2^{1/2+\epsilon} X^{-1/2} Y^{-1/2}.
\end{align}
Finally, for $T_3$ we estimate
\begin{equation}
T_3 \ll \sum_{q_1\in \mc D}q_1 \underset{n \leq (q_1q_2)^{1+\epsilon}/Y}{\sum_{m \leq q_1^{1+\epsilon}/X}}\frac{|\lambda_f(m)\lambda_g(n)|}{(mn)^{1/2}} q_1^{\frac{1}{2}} \ll  Q_1^{7/2+\epsilon} q_2^{1/2+\epsilon} X^{-1/2} Y^{-1/2}.
\end{equation}
Thus combining everything we get 
\begin{equation} \label{finalSnn}
  \mbf  S_{nn} \ll  Q_1^{3+\epsilon}q_2 X^{-1/2}+Q_1^{7/2+\epsilon} q_2^{1/2+\epsilon} X^{-1/2} Y^{-1/2}.
\end{equation}
\noindent \textbf{Proof of \thmref{thasymp}}: For the simplicity of collecting the error terms, we put $q_2 \asymp Q_1^{c}$, where $c>0$ is a parameter depending on $Q_1, q_2$. From the considerations in sections \ref{sec:sdd} and \ref{sec:snd}, we see that $c\ge 1$. Therefore, we note that $\pi_2(Q_1)q_2\asymp Q_1^{3+c}/3(\log 2Q_1)$. Hence, from the estimate in \eqref{finalSnn}, we observe that $X=Q_1^{\alpha}$ for some $\alpha>0$ would suffice. We also note from the estimates in  \eqref{finalSdn}, \eqref{finalSdd}, \eqref{finalSnd}, \eqref{finalSnn} that it suffices to take $Y=q_2^{-\beta}$ for some $\beta> 0$. Thus, combining everything we get
\begin{align} \label{semifinalS}
\mbf S&= \frac{\lambda_g(p)}{p^{1/2}} \pi_2(Q_1) q_2 + 
   O(Q_1^{\frac{5}{2} + \alpha +\epsilon}q_{2} + Q_1^{3+ \frac{\alpha}{2} +\epsilon}q_2^{\frac{1+\beta}{2}+\epsilon} + Q_1^{\frac{7+ \alpha}{2} +\epsilon}q_2^{\frac{1+2\beta}{2}+\epsilon}) \\
   &+ O( Q_1^{\frac{11-\alpha}{4}} q_2 +Q_1^{4+ \alpha +\epsilon} + Q_1^{\frac{11}{4}} q_2^{\frac{3-\beta}{4}} + Q_1^{3+\frac{\alpha}{2}+\epsilon}q_2^{1-\frac{\beta}{2}+\epsilon} +Q_1^{\frac{7+ \alpha}{2} +\epsilon}q_2^{\frac{3}{4}+\epsilon}) \\
   &+O(Q_1^{3-\frac{\alpha}{2}+\epsilon} q_2^{1-\frac{\beta}{2}+\epsilon} + Q_1^{\frac{7-\alpha}{2}+\epsilon} q_2^{\frac{1+\beta}{2}+\epsilon}  + Q_1^{3-\frac{\alpha}{2}+\epsilon}q_2 ).
\end{align}
We immediately note that $Q_1^{\frac{7+ \alpha}{2} +\epsilon}q_2^{\frac{1+2\beta}{2}+\epsilon} \gg Q_1^{3+ \frac{\alpha}{2} +\epsilon}q_2^{\frac{1+\beta}{2}+\epsilon} $ and $Q_1^{3+\frac{\alpha}{2}+\epsilon}q_2^{1-\frac{\beta}{2}+\epsilon} \gg Q_1^{3-\frac{\alpha}{2}+\epsilon} q_2^{1-\frac{\beta}{2}+\epsilon} $. Let $c_0$ denote the greatest lower bound for all such  $c$. Now, since our effort is to make $c_0$ as small as possible, the error term $Q_1^{\frac{7+ \alpha}{2} +\epsilon}q_2^{\frac{3}{4}+\epsilon}$ indicates that smaller the $\alpha$ is chosen, smaller will be the $c_0$. To keep the optimization process simple, we choose $\alpha=1/100$. Now comparing the main term with the error term $Q_1^{\frac{7+ \alpha}{2} +\epsilon}q_2^{\frac{3}{4}+\epsilon}$, we see that $c>101/50+\epsilon$. We fix $c \geq 101/50+1/100= 203/100$.

 As $\alpha=1/100$, we note that $Q_1^{3-\frac{\alpha}{2}+\epsilon}q_2 \gg  Q_1^{\frac{11-\alpha}{4}} q_2 \gg Q_1^{\frac{5}{2} + \alpha +\epsilon}q_{2} $. Further as $c_{0}= 203/100$, we observe that $Q_1^{3-\frac{\alpha}{2}+\epsilon}q_2 \gg Q_1^{4+ \alpha +\epsilon} $ and $Q_1^{3-\frac{\alpha}{2}+\epsilon}q_2 \gg Q_1^{\frac{11}{4}} q_2^{\frac{3-\beta}{4}} $.
 
 Now to choose $\beta$, taking $p \ll \log Q_1$, we compare the main term in \eqref{semifinalS} with the error term $Q_1^{3+\frac{\alpha}{2}+\epsilon}q_2^{1-\frac{\beta}{2}+\epsilon}$ and get
 \begin{equation} \label{ffirstfinalS}
    \frac{Q_1^{3+c }}{(\log 2Q_1)^{3/2}} \gg  Q_1^{3+\frac{\alpha}{2}+c-\frac{\beta c}{2}+ \epsilon}.
\end{equation}
This implies $\beta c > 1/100 +\epsilon$. As $c_0= 203/100$, we need $\beta >1/203+ \epsilon$. We choose $\beta=1/100$. Thus we calculate that $Q_1^{3-\frac{\alpha}{2}+\epsilon}q_2 \gg Q_1^{\frac{7+ \alpha}{2} +\epsilon}q_2^{\frac{1+2\beta}{2}+\epsilon}$, $Q_1^{3-\frac{\alpha}{2}+\epsilon}q_2 \gg Q_1^{\frac{7-\alpha}{2}+\epsilon} q_2^{\frac{1+\beta}{2}+\epsilon} $ and $Q_1^{3-\frac{\alpha}{2}+\epsilon}q_2 \gg Q_1^{3+\frac{\alpha}{2}+\epsilon}q_2^{1-\frac{\beta}{2}+\epsilon} $. Hence it remains to compare only two error terms $Q_1^{3-\frac{\alpha}{2}+\epsilon}q_2$ and $Q_1^{\frac{7+ \alpha}{2} +\epsilon}q_2^{\frac{3}{4}+\epsilon}.$

We calculate that for $203/100 \leq c \leq 51/25 $, we have 
\begin{equation*}
    Q_1^{\frac{7+ \alpha}{2} +\epsilon}q_2^{\frac{3}{4}+\epsilon} \gg Q_1^{3-\frac{\alpha}{2}+\epsilon}q_2
\end{equation*}
and for $c > 51/25$, we have
\begin{equation*}
    Q_1^{\frac{7+ \alpha}{2} +\epsilon}q_2^{\frac{3}{4}+\epsilon} \ll Q_1^{3-\frac{\alpha}{2}+\epsilon}q_2.
\end{equation*}
This completes the proof.\qed

\section{Proofs of Theorems \ref{unipair} and \ref{thsimul}}

\subsection{ Proof of \thmref{unipair}}
It is enough to consider the pairs $(Q_1, q_2)$ such that $q_2 = Q_1^{5/2}+O(Q_1^2)$ (see Remark \ref{remarksizeq2}). Thus after putting $q_2 = Q_1^{5/2}+O(Q_1^2)$ in  \thmref{thasymp} we get that
\begin{equation} \label{newthasymp}
    \sum_{q_1\in \mc D}\sideset{}{^*}\sum_{\chi(q_1)}\sideset{}{^*}\sum_{\psi(q_2)}\overline{\chi}(p)\overline{\psi}(p)\overline{G(\chi)}^2L(1/2, f\times \chi)L(1/2, g\times \chi\psi)= \frac{\lambda_g(p)}{p^{1/2}}\cdot \pi_2 (Q_1) Q_1^{\frac{5}{2}}  +O(Q_1^{\frac{1099}{200} +\epsilon}).
\end{equation}
Now, for $i=1,2$ if we have $f_i\in S_{k_i}(N_i)$ and $g_i\in S_{l_i}(M_i)$ such that
\begin{equation} \label{fg}
    L(1/2, f_1\times \chi)L(1/2, g_1\times \chi\psi)=L(1/2, f_2\times \chi)L(1/2, g_2\times \chi\psi),
\end{equation} 
we invoke \eqref{newthasymp} and get
\begin{equation} \label{equalg}
     \frac{\lambda_{g_1}(p)}{p^{1/2}}\cdot \pi_2 (Q_1) Q_1^{\frac{5}{2}}  +O(Q_1^{1099/200+\epsilon})=  \frac{\lambda_{g_2}(p)}{p^{1/2}}\cdot \pi_2 (Q_1) Q_1^{\frac{5}{2}} +O(Q_1^{1099/200+\epsilon}).
\end{equation}
Let us recall that $\pi_2(Q_1)Q_1^{\frac{5}{2}} \asymp Q^{\frac{11}{2}}/\log Q_1$. As \eqref{equalg} holds for every large enough $Q_1$ and all $p$ such that $(p, M_1M_2)=1$ and $p \leq \log Q_1$, it implies that
$\lambda_{g_1}(p)= \lambda_{g_2}(p)$ for all $p$ co-prime with $M_1M_2$. Thus from the strong multiplicity-one property for newforms we have $g_1= g_2$, $l_1=l_2$ and $M_1=M_2$.

Now to prove $f_1=f_2$ from here, we need that for any fixed primitive $\chi$ there exists at least one primitive $\psi$ such that $L(1/2, g \times \chi \psi)$ is non-zero. We will do this by using the \lemref{lemma} below, which is well known to the mathematicians working in this field. However, an asymptotic with a known dependency on the level of the modular form is not readily available in the literature. Thus, for the sake of completeness and convenience of the reader, we write a proof here.
\begin{lem} \label{lemma}
    Let $h$ be a newform of fixed level $q_h$ and $p$ be a prime with $(p, q_h)=1$. Further, let $D$ be a set of primes defined as 
    \[D:= \{q \text{ is prime} | Q\leq q \leq 2Q \} .\]
    Then for any fixed prime $p < Q$ with $(p,q_h)=1$, we have
    \begin{equation} \label{nonvang}
   W:= \sum_{ q \in D}\sideset{}{^*}\sum_{\psi (q)} \overline{\psi} (p) L(1/2, h\times \psi)= \frac{\lambda_h(p) }{p^{1/2}}\cdot \pi_1(Q) +O(Q^{ 7/4 +\epsilon} q_h^{1/4+\epsilon}).        
    \end{equation}
\end{lem}
\begin{proof}
    Using the approximate functional equation from \eqref{AFE} in \eqref{nonvang}, we write the sum for non-dual and dual part of $W$ as $W_{nd}$ and $W_d$ respectively. So we get
    \begin{equation} \label{Wnd}
    W_{nd}:= \sum_{ q \in D}\sideset{}{^*}\sum_{\psi(q)} \sum_{n\geq 1} \frac{\lambda_{h}(n) \psi(n \bar{p})}{n^{1/2}} V_h\Big(\tfrac{nX}{q_{h \times \psi}^{1/2}}\Big).
    \end{equation}
The rapid decay of $V_h(y)$ (see \eqref{Vbound}) allows us to truncate the sum over $n$ at $n\ll_{\epsilon} q_h^{1/2} q^{1+\epsilon}/X$ with negligible error. Thus we calculate 
\begin{equation} \label{Wndcal}
    \begin{split}
        W_{nd}&= \sum_{ q \in D} \sum_{n\geq 1} \frac{\lambda_{h}(p) }{n^{1/2}} \phi(q) \delta_{n \equiv p \bmod{q}} V_h\Big(\tfrac{nX}{q_{h \times \psi}^{1/2}}\Big) + O(Q^{3/2 +\epsilon} q_h^{1/4+\epsilon} X^{-1/2}) \\
        &= \sum_{ q \in D} \frac{\lambda_{h}(p)}{p^{1/2}} \phi(q) + \sum_{ q \in D}\phi(q) \sum_{ q < n \leq q_h^{1/2} q^{1+\epsilon}/X} \frac{\lambda_{h}(n)}{n^{1/2}}  \delta_{n \equiv p \bmod{q}} V_h\Big(\tfrac{nX}{q_{h \times \psi}^{1/2}}\Big) \\
        & +O(Q^{\frac{7}{4}}q_h^{\frac{1}{8}} X^{\frac{1}{4}}+Q^{\frac{3}{2} +\epsilon} q_h^{\frac{1}{4}+ \epsilon} X^{\frac{-1}{2}}) \\
        &= \frac{\lambda_{h}(p)}{p^{1/2}} \pi_1(Q) + O\Big(\sum_{ q \in D} \sum_{1\leq r \leq q_h^{1/2}q^{\epsilon}/X} \frac{\phi(q)}{(p+rq)^{\frac{1}{2}- \epsilon}}  \Big) + O(Q^{\frac{3}{2} +\epsilon} q_h^{\frac{1}{4}+\epsilon} X^{\frac{-1}{2}})\\
        & +O(Q^{\frac{7}{4}}q_h^{\frac{1}{8}} X^{\frac{1}{4}}+Q^{\frac{3}{2} +\epsilon} q_h^{\frac{1}{4}+ \epsilon} X^{\frac{-1}{2}}) \\
        &= \frac{\lambda_{h}(p)}{p^{1/2}} \pi_1(Q) +O(Q^{\frac{7}{4}}q_h^{\frac{1}{8}} X^{\frac{1}{4}}+Q^{\frac{3}{2} +\epsilon} q_h^{\frac{1}{4}+ \epsilon} X^{\frac{-1}{2}}).
    \end{split}
\end{equation}
For the dual part of $W$, we calculate
\begin{equation} \label{Wdcal}
    \begin{split}
        W_d &= \sum_{ q \in D}\sideset{}{^*}\sum_{\psi(q)} \frac{G( \psi)^2 }{ q} \sum_{n\geq 1} \frac{\overline{\lambda_{h}}(n) \overline{\psi}(n p)}{n^{1/2}} V_h\Big(\tfrac{n}{Xq_{h \times\psi}^{1/2}}\Big) \\
        &= \sum_{ q \in D} \frac{1 }{ q} \sum_{n\geq 1} \frac{\overline{\lambda_{h}}(n) }{n^{1/2}} \Big(\sideset{}{^*}\sum_{\psi(q)} G(\psi)^2 \overline{\psi} (np ) \Big) V_h\Big(\tfrac{n}{Xq_{h \times\psi}^{1/2}}\Big) \\
        &= \sum_{ q \in D} \frac{\phi(q)}{ q} \sum_{1 \leq n \leq q_h^{1/2}q^{1+\epsilon}X} \frac{\overline{\lambda_{h}}(n) }{n^{1/2}} S(1, np; q) V_h\Big(\tfrac{n}{q_{h \times \psi}^{1/2}}\Big) + O(Q^{\frac{1}{2}+\epsilon} q_h^{\frac{1}{4}+\epsilon} X^{\frac{1}{2}})\\
        &= O(Q^{2+\epsilon} q_h^{\frac{1}{4}+\epsilon} X^{\frac{1}{2}}).
    \end{split}
\end{equation} 
Now we put $X= Q^{-1/2}$ in \eqref{Wndcal} and \eqref{Wdcal} respectively and combine them to get \eqref{nonvang}.
\end{proof}
 By putting $h= g \times \chi$ and taking $p$ to be such that $\lambda_{g}(p) \neq 0$ we easily conclude from \lemref{lemma} that for any fixed primitive $\chi$ of conductor $q_1$, there exists some primitive $\psi$ with conductor $q_2 \gg q_1^{5/2}$ such that $L(1/2, g \times \chi \psi)$ is non-zero. 

Now from \eqref{fg} we infer that for every primitive $\chi$ of prime conductor, for all sufficiently large prime  we have
\begin{equation} \label{f=f}
   L(1/2, f_1\times \chi) =L(1/2, f_2\times \chi). 
\end{equation}
To prove $f_1=f_2$ from \eqref{f=f}, we put $\chi=\psi$ and again apply \lemref{lemma} to both of $f_1$ and $f_2$. Thus it gives that $\lambda_{f_1}(p)= \lambda_{f_2}(p)$ for almost all primes $p$ with $(p,N_1N_2)=1$. Now from the strong multiplicity-one theorem we conclude that $k_1=k_2$, $N_1=N_2$ and $f_1=f_2$. \qed
\begin{rmk} \label{remarksizeq2}
From Theorem.$1$ in \cite{baker2001difference} we know that there is an $x_0 >0$ such that for any $x >x_0$ there is a prime in the interval $[x-x^{0.525},x]$. Thus for any sufficiently large $Q_1$, one can find a prime $q_2$ satisfying $q_2=Q_1^{5/2}+O(Q_1^2)$ .
 \end{rmk}
\subsection{Simultaneous non-vanishing}We recall the set $\mc M$ from \eqref{nonvanset}. Then using the Cauchy-Schwarz inequality, we get
\begin{align}\label{CSsimul}
 &\left|\sum_{q_1\in \mc D}\sideset{}{^*}\sum_{\chi(q_1)}\sideset{}{^*}\sum_{\psi(q_2)}\overline{\chi}(p)\overline{\psi}(p)\overline{G(\chi)}^2L(1/2, f\times \chi)L(1/2, g\times \chi\psi)\right|^2 \nonumber\\
 &\le Q_1^2\cdot \#\mc M\cdot \sum_{q_1\in \mc D}\frac{q_1}{\phi(q_1)}\sideset{}{^*}\sum_{\chi(q_1)}\sideset{}{^*}\sum_{\psi(q_2)}\left|L(1/2, f\times \chi)L(1/2, g\times \chi\psi)\right|^2.
 \end{align}
We will estimate the right-hand side by using the approximate functional equation and the large sieve inequality. First, using the approximate functional equation (see \eqref{AFE}) we write
\begin{equation}
   L(1/2, f\times \chi)L(1/2, g\times \chi\psi)= \mc S_1+\mc S_2+\mc S_3+\mc S_4,
\end{equation}
where
\begin{align*}  
\mc S_{1}&=\sum_{m,n} \chi(mn)\psi(n)\frac{\lambda_f(m)\lambda_g(n)}{(mn)^{1/2}}V_f\Big(\tfrac{mX}{q_\chi^{1/2}}\Big)V_g\Big(\tfrac{nY}{q_{\chi\psi}^{1/2}}\Big),\\ 
    \mc S_{2}&= \frac{G(\chi\psi)^2}{q_1q_2}\sum_{m,n} \chi(m\bar{n})\overline{\psi}(n)\frac{\lambda_f(m)\lambda_g(n)}{(mn)^{1/2}}V_f\left(\tfrac{mX}{q_\chi^{1/2}}\right)V_g\left(\tfrac{n}{Yq_{\chi\psi}^{1/2}}\right),\\
    \mc S_{3}&=\frac{G(\chi)^2}{q_1}\sum_{m,n}  \chi(\bar{m}n) \psi(n)\frac{\lambda_f(m)\lambda_g(n)}{(mn)^{1/2}}V_f\Big(\tfrac{m}{Xq_\chi^{1/2}}\Big)V_g\Big(\tfrac{nY}{q_{\chi\psi}^{1/2}}\Big),\\
    \mc S_{4}&= \frac{G(\chi\psi)^2G(\chi)^2}{q_1^2q_2}\sum_{m,n} \overline{\chi}(mn)\overline{\psi}(n)\frac{\lambda_f(m)\lambda_g(n)}{(mn)^{1/2}}V_f\Big(\tfrac{m}{Xq_\chi^{1/2}}\Big)V_g\Big(\tfrac{n}{Yq_{\chi\psi}^{1/2}}\Big).
\end{align*}
Thus
\begin{equation}
    \begin{split}
    &\sum_{q_1\in \mc D}\frac{q_1}{\phi(q_1)}\sideset{}{^*}\sum_{\chi(q_1)}\sideset{}{^*}\sum_{\psi(q_2)}\left|L(1/2, f\times \chi)L(1/2, g\times \chi\psi)\right|^2 \\ &\ll \sum_{i=1}^{4}\sum_{q_1\in \mc D}\frac{q_1}{\phi(q_1)}\sideset{}{^*}\sum_{\chi(q_1)}\sideset{}{^*}\sum_{\psi(q_2)} {|\mc S_i|^2}\ll \sum_{i=1}^{4}\sum_{d_1\ll Q_1 }\frac{d_1}{\phi(d_1)}\sideset{}{^*}\sum_{\chi(d_1)}\sideset{}{^*}\sum_{\psi(q_2)} {|\mc S_i|^2}.
    \end{split}
\end{equation}
First we put $X=Y=1$ and note that the sums over $m,n$ in $\mc S_i$ can be truncated at $\ll_{k_1, N_1} M:= Q_1^{1+\epsilon}$ and $\ll_{k_2, N_2} N:= (Q_1q_2)^{1+\epsilon}$ respectively with negligible error in $Q_1$ and $q_2$. Using the large sieve inequality (see \cite[Theorem 7.13]{iwaniec2021analytic}) we get
\begin{equation}
    \begin{split}
        \sum_{d_1}\frac{d_1}{\phi(d_1)}\sideset{}{^*}\sum_{\chi(d_1)}\sideset{}{^*}\sum_{\psi(q_2)} {|\mc S_1|^2}&\ll (q_2+ N)\sum_{d_1\ll Q_1}\frac{d_1}{\phi(d_1)}\sideset{}{^*}\sum_{\chi(d_1)}\sum_{n\ll N}\frac{|\lambda_g(n)|^2}{n}\left|\sum_{m\ll M} \chi(m)\tfrac{\lambda_f(m)}{(m)^{1/2}}V_f\right|^2\\
        &\ll (q_2+N)(Q_1^2+M)\sum_{n\ll N}\sum_{m\ll M}\left|\tfrac{\lambda_f(m)\lambda_g(n)}{(mn)^{1/2}}\right|^2\\
        &\ll (q_2+N)(Q_1^2+M) \log N \log M\\
        &\ll Q_1^{2} q_2^{1+1/c+\epsilon}\log Q_1.
    \end{split}
\end{equation}
In the last inequality we have used that $q_2 \asymp Q_1^c$ and that $q_2+N\ll N \ll q_2^{1+1/c+\epsilon}$, $Q_1^2+M\ll Q_1^2$. We further note that the same bound holds for $\mc S_2$, $\mc S_3$, and $\mc S_4$. Thus we get
\begin{equation}
    \sum_{q_1\in \mc D}\frac{q_1}{\phi(q_1)}\sideset{}{^*}\sum_{\chi(q_1)}\sideset{}{^*}\sum_{\psi(q_2)}\left|L(1/2, f\times \chi)L(1/2, g\times \chi\psi)\right|^2  \ll Q_1^{2} q_2^{1+1/c+\epsilon}\log Q_1.
\end{equation}
Thus from \eqref{CSsimul}, using the asymptotic from Theorem \ref{thasymp} we conclude that
\begin{equation}
    \# \mc M\gg \frac{|\lambda_g(p)|^2}{p}\cdot \frac{Q_1^6(\log Q_1)^{-2}q_2^2}{Q_1^{4} \log Q_1q_2^{1+1/c+\epsilon}}\gg \frac{|\lambda_g(p)|^2}{p}Q_1^{2}(\log Q_1)^{-3}q_2^{1-1/c-\epsilon}.
\end{equation}
Now from the prime number theorem of $GL(2)$ $L$-functions (see \cite[Theorem 5.13]{iwaniec2021analytic}), we have 
\begin{equation}
    \sum_{p\le x} \lambda_g(p)^2 \log p = x+O( k_2N_2^{1/2}x\exp(-c\sqrt{\log x})).
\end{equation}
Thus we can choose a $p$ independent of $Q_1$ and $q_2$ such that $\lambda_g(p)^2\geq C$, for some absolute constant $C>0$. 
This completes the proof of \thmref{thsimul}.\qed
\printbibliography
\end{document}